\documentclass[12pt]{amsart}
\usepackage{exscale,aliascnt,amssymb,amsthm,amsxtra,latexsym,mathrsfs,mathtools,paralist,varioref,verbatim}
\usepackage[mathscr]{eucal}
\usepackage[T1]{fontenc}
\usepackage[margin=1in]{geometry}

\usepackage[pagebackref]{hyperref} 

\usepackage[capitalize]{cleveref} 

\usepackage[abbrev,backrefs,msc-links]{amsrefs} 

\DefineSimpleKey{bib}{article-id}
\DefineSimpleKey{bib}{category}

\newcommand{\arxiv}[1]{\href{http://arxiv.org/abs/math/#1}{arXiv:#1 [\BibField{category}]}}

\BibSpec{arxiv}{%
    +{}  {\PrintAuthors}                {author}
    +{,} { \textit}                     {title}
    +{}  { \PrintDatePV}                {date}
    +{,} { \eprintpages}                {pages}
    +{,} { \arxiv}                      {article-id}
    +{.} { }                            {note}
    +{.} {}                             {transition}
}

\usepackage[all]{xy} 

\theoremstyle{plain}
\newtheorem{thm}{Theorem}

\newaliascnt{cor}{thm}

\aliascntresetthe{cor}
\newaliascnt{lem}{thm}
\newtheorem{lem}[lem]{Lemma}
\aliascntresetthe{lem}
\newaliascnt{prop}{thm}

\aliascntresetthe{prop}

\theoremstyle{definition}
\newaliascnt{defn}{thm}
\newtheorem{defn}[defn]{Definition}
\aliascntresetthe{defn}

\theoremstyle{remark}

\DeclareMathOperator{\convhull}{Conv}

\DeclareMathOperator{\fitt}{Fitt}
\DeclareMathOperator{\interior}{int}

\DeclareMathOperator{\newt}{Newt}

\DeclareMathOperator{\rank}{rank}

\DeclareMathOperator{\spec}{Spec}

\newcommand*{\jac}{\mathscr{J}ac}
\newcommand*{\logjac}{\mathscr{J}ac^{log}}
\newcommand*{\monoid}{\mathsf{S}}
\newcommand*{\sheaf}[1]{\mathscr#1} 

\newcommand*{\multideal}{\sheaf{J}}

\newcommand*{\NN}{\mathbb{N}}

\newcommand*{\RR}{\mathbb{R}}
\newcommand*{\ZZ}{\mathbb{Z}}

\renewcommand*{\vec}{\mathbf}

\begin{document}

\title{Mather-Jacobian Multiplier Ideals on Toric Varieties}
\date{\today}
\author{Howard M Thompson}
\address{Mathematics Department\\ 402 Murchie Science Building \\ 303 East Kearsley Street \\ Flint, MI 48502-1950}
\email{\href{mailto:hmthomps@umflint.edu}{hmthomps@umflint.edu}}
\urladdr{\url{http://homepages.umflint.edu/~hmthomps/}}
\thanks{}
\begin{abstract}
  This paper provides a formula for the Mather-Jacobian multiplier ideals of torus invariant ideals on (not necessarily normal) toric varieties that generalizes Howald’s formula for the multiplier ideal of monomial ideals in a polynomial ring.
\end{abstract}
\maketitle

\section{Introduction}

In \vref{s:nash}, we recall the definition of Nash blowup and the relationship between Nash blowups and Jacobian ideals. In \vref{s:logjac}, we recall the definition of logarithmic Jacobian ideals and the relationships between Nash blowups, logarithmic Jacobian ideals, and Jacobian ideals. In \vref{s:MJ}, we recall the definition of Mather-Jacobian multiplier ideals and present a formula for the Mather-Jacobian multiplier ideals of torus invariant ideals on (not necessarily normal) toric varieties.

\section{Nash Blowup and the Jacobian} \label{s:nash}

In this section we loosely follow de Fernex \& Docampo~\cite[Section~2]{MR3141731}.

Let $X$ be a separated reduced scheme of finite type over an algebraically closed field $\Bbbk$ of characteristic zero. The \emph{Jacobian ideal sheaf} of $X$, $\jac_X\subseteq\sheaf{O}_X$, is the smallest non-zero Fitting ideal of the cotangent sheaf $\Omega_X$. If $X$ is of pure dimension $d$, then $\jac_X=\fitt^d(\Omega_X)$ is the $d$-th Fitting ideal of $\Omega_X$. If $Y$ is a smooth scheme and $f:Y\rightarrow X$ is a morphism of schemes, the Jacobian ideal of $f$ is the $0$-th Fitting ideal $\jac_f=\fitt^0(\Omega_{Y/X})$ of $\Omega_{Y/X}$. If $Y$ is equidimensional of dimension $d$, then $\Omega_Y^d$ is an invertible sheaf and the image of the map $f^*\Omega_X^d\rightarrow\Omega_Y^d$ is $\jac_f\cdot\Omega_Y^d$.

Let $X$ be a separated reduced scheme of finite type over $\Bbbk$ of pure dimesion $d$. The \emph{Nash blowup} of $X$ is the birational map $\pi:\widehat{X}\rightarrow X$ satisfying the universal property: if $\mu:Y\rightarrow X$ is a proper birational map such that $\mu^*\Omega^1_X/(torsion)$ is a locally free $\sheaf{O}_Y$-module, then $\mu$ factors uniquely through $\pi$.  Following Ein, Ishii \& Musta{\c{t}}{\u{a}}~\cite[Remark~2.3]{EIM2011} and noticing that $X$ need not be a variety to apply Lipman~\cite[Lemma~1]{MR0237511}, we see every log resolution of $\jac_X$ factors through the Nash blowup of $X$.  In fact, by Oneto \& Zatini~\cite[Corollary~5.2]{MR1218672}, if $\jac_{Y}|_X$ is the restriction to $X$ of the Jacobian of a locally complete intersection $Y$ of dimension $d$ containing $X$, then the Nash blowup is the blowup of $\jac_{Y}|_X$.

\section{The Jacobian and the Logarithmic Jacobian} \label{s:logjac}

From now on, let $\monoid$ be an affine semigroup (finitely generated torsion-free cancellative monoid) and let $M=\ZZ\monoid\cong\ZZ^d$ be its quotient group. So, $X=\spec\Bbbk[\monoid]$ is a (not necessarily normal) affine toric variety. We write $\vec{m}$ for an element of $M_\RR=M\otimes_\ZZ\RR$ and we write $\chi^\vec{m}$ for the element of $\Bbbk[M]$ corresponding to $\vec{m}\in M$. We have a $\Bbbk[\monoid]$-module homomorphism
\[
  \varphi:\Omega^1_{\Bbbk[\monoid]/\Bbbk}\rightarrow\Bbbk[\monoid]\otimes_\ZZ M;\quad\varphi(d\chi^\vec{m})=\chi^\vec{m}\otimes\vec{m}.
\]

\begin{defn}
  The \emph{logarithmic Jacobian ideal} $\logjac_X$ of $X$ is the ideal such that the image of $\wedge^d\varphi$ is $\logjac_X\otimes_\ZZ\wedge^d M$. This ideal is generated by elements $\chi^{\vec{m}_1+\vec{m}_2+\cdots+\vec{m}_d}$ such that $\vec{m}_1\wedge\vec{m}_2\wedge\cdots\wedge\vec{m}_d\neq\vec{0}$. Evidently, it suffices to take the $\vec{m}_i$ from some fixed generating set for $\monoid$.
\end{defn}

According to Gonz{\'a}lez P{\'e}rez \& Teissier~\cite[Proposition~60]{MR3183106}, the Nash blowup $\pi:\widehat{X}\rightarrow X$ of $X$ is the blowup of $\logjac_X$. Examining the proof of this proposition provides a convenient formula for the Jacobian of $X$. Fix a surjective monoid homomorphism $\phi:\NN^r\rightarrow\monoid$. This surjection induces a short exact sequence of Abelian groups
\[
  \xymatrix{
    0 \ar[r] & \mathscr{L} \ar[rr] & & \ZZ^r \ar[rr]^{\phi^{gp}} & & M \ar[r] & 0
  }.
\]
For each element $\vec{u}\in\ZZ^r$, write $\vec{u}=\vec{u}^+-\vec{u}^-$ where $\vec{u}^+$ and $\vec{u}^-$ are elements of $\NN^r$ with disjoint support.  The monoid homomorphism $\phi$ induces a $\Bbbk$-algebra homomorphism $\phi^{alg}:\Bbbk[x_1,x_2,\ldots,x_r]\rightarrow\Bbbk[\monoid]$ with kernel $I_\mathscr{L}=\left(\vec{x}^{\vec{u}^+}-\vec{x}^{\vec{u}^-}\mid\vec{u}\in\mathscr{L}\right)$. Set $c=\rank\mathscr{L}=r-d$, let $\vec{e}_1,\vec{e}_2,\ldots\vec{e}_r$ be the standard basis in $\NN^r$, let $\vec{1}=\sum_{i=1}^r\vec{e}_i$ be the all ones vector, and fix a generating set $\left\{f_1=\vec{x}^{\vec{u}_1^+}-\vec{x}^{\vec{u}_1^-},f_2=\vec{x}^{\vec{u}_2^+}-\vec{x}^{\vec{u}_2^-},\ldots,f_n=\vec{x}^{\vec{u}_n^+}-\vec{x}^{\vec{u}_n^-}\right\}$ for $I_\mathscr{L}$. 

Let $J$ be the $r\times n$ matrix $\begin{bmatrix}\frac{\partial f_j}{\partial x_i}\end{bmatrix}$, let $U$  be the $r\times n$ matrix with columns $\vec{u}_1,\vec{u}_2,\ldots,\vec{u}_n$, let $K=\{k_1,k_2,\ldots,k_c\}\subseteq\{1,2,\ldots,r\}$, and let $L=\{j_1,j_2,\ldots,j_c\}\subseteq\{1,2,\ldots,n\}$. Gonz{\'a}lez P{\'e}rez \& Teissier prove following relationship between the $c\times c$ minors $J_{KL}$ and $U_{KL}$.
\begin{equation} \label{eq:minors}
  \vec{x}^{\vec{e}_{k_1}+\vec{e}_{k_2}+\cdots\vec{e}_{k_c}}\cdot J_{KL}\equiv\vec{x}^{\vec{u}_{j_1}^++\vec{u}_{j_2}^++\cdots+\vec{u}_{j_c}^+}U_{KL} \mod I_\mathscr{L}.
\end{equation}

Now, fix $1\leq j_1<j_2<\cdots<j_c\leq n$ such that $\vec{0}\neq\vec{u}_{j_1}\wedge\vec{u}_{j_2}\wedge\cdots\wedge\vec{u}_{j_c}\in\wedge^c\mathscr{L}$. Let $I=\left(\vec{x}^{\vec{u}_{j_1}^+}-\vec{x}^{\vec{u}_{j_1}^-},\vec{x}^{\vec{u}_{j_2}^+}-\vec{x}^{\vec{u}_{j_2}^-},\ldots,\vec{x}^{\vec{u}_{j_c}^+}-\vec{x}^{\vec{u}_{j_c}^-}\right)\subseteq\Bbbk[\vec{x}]$ Then, $Y=\spec\left(\Bbbk[\vec{x}]/I\right)$ is a reduced complete intersection of dimension $d$ containing $X$. After multiplying both sides of the congruence \vref{eq:minors} by $\vec{x}^{\vec{1}-(\vec{e}_{k_1}+\vec{e}_{k_2}+\cdots\vec{e}_{k_c})}$, varying $K$, and considering the restriction to $X$, one obtains
\[
  \chi^{\phi^{gp}(\vec{1})}\cdot\jac_{Y|X}=\chi^{\phi^{gp}(\vec{u}_{j_1}^++\vec{u}_{j_2}^++\cdots+\vec{u}_{j_c}^+)}\cdot\logjac_X.
\]

\begin{lem}
  $\jac_X$ is generated by
  \[
    \left\{\chi^{\phi^{gp}(\vec{e}_{i_1}+\vec{e}_{i_2}+\cdots+\vec{e}_{i_d}+\vec{u}_{j_1}^++\vec{u}_{j_2}^++\cdots+\vec{u}_{j_c}^+-\vec{1})}\,\left|\,\begin{matrix}
    \vec{0}\neq\phi^{gp}(\vec{e}_{i_1})\wedge\phi^{gp}(\vec{e}_{i_2})\wedge\cdots\wedge\phi^{gp}(\vec{e}_{i_d})\in\wedge^d M; \\
    \vec{0}\neq\vec{u}_{j_1}\wedge\vec{u}_{j_2}\wedge\cdots\wedge\vec{u}_{j_c}\in\wedge^c\mathscr{L}
    \end{matrix}\right.\right\}.
  \]
\end{lem}

\begin{proof}
  Continuing as above, allow $L$ to vary as well. We deduce
  \[
    \chi^{\phi^{gp}(\vec{1})}\cdot\jac_X=\logjac_X\cdot\left(\left.\chi^{\phi^{gp}(\vec{u}_{j_1}^++\vec{u}_{j_2}^++\cdots+\vec{u}_{j_c}^+)}\,\right|\,\vec{0}\neq\vec{u}_{j_1}\wedge\vec{u}_{j_2}\wedge\cdots\wedge\vec{u}_{j_c}\in\wedge^c\mathscr{L}\right).
  \]
\end{proof}

\section{Mather-Jacobian Multiplier Ideals} \label{s:MJ}

\begin{defn}
  Let $\mathfrak{a}\subseteq\sheaf{O}_X$ be a nonzero ideal sheaf on a variety $X$, let $\mu:Y\rightarrow X$ be a log resolution of $\mathfrak{a}\cdot\jac_X$, let $\sheaf{O}_Y\left(-\widehat{K}_{Y/X}\right)=\jac_\mu$, let $\sheaf{O}_Y(-J_{Y/X})=\jac_X\cdot\sheaf{O}_Y$, let $\sheaf{O}_Y(-Z_{Y/X})=\mathfrak{a}\cdot\sheaf{O}_Y$, and let $\lambda\in\RR_{\geq0}$. The \emph{Mather-Jacobian multiplier ideal} of $\mathfrak{a}$ with exponent $\lambda$ is
  \[
    \widehat{\multideal}(X,\mathfrak{a}^\lambda)=\mu_*\sheaf{O}_Y\left(\widehat{K}_{Y/X}-J_{Y/X}-\lfloor\lambda Z_{Y/X}\rfloor\right).
  \]
\end{defn}

Now, we are ready to develop a formula for the Mather-Jacobian multiplier ideals of torus invariant ideals on (not necessarily normal) toric varieties. it suffice to consider the affine case. Let $X=\spec\Bbbk[\monoid]$ be a (not necessarily normal) affine toric variety, let $\mathfrak{a}\subseteq\Bbbk[\monoid]$ be a nonzero monomial ideal, let $\mu:X_\Sigma\rightarrow X$ be a toric log resolution of $\mathfrak{a}\cdot\jac_X$. 

\begin{lem} \label{lem:k-hat}
  If $\mathfrak{a}\subseteq\sheaf{O}_X$ is a nonzero invariant ideal sheaf on a toric variety $X$ and $\mu:X_\Sigma\rightarrow X$ is a toric log resolution of $\mathfrak{a}\cdot\jac_X$, then
  \[
    \jac_\mu\cdot\logjac_{X_\Sigma}=\mu^*\logjac_X\text{ and }\sheaf{O}_{X_\Sigma}\left(\widehat{K}_{X_\Sigma/X}\right)=\sheaf{Hom}_{\sheaf{O}_{X_\Sigma}}\left(\mu^*\logjac_X,\logjac_{X_\Sigma}\right).
  \]
\end{lem}

\begin{proof}
  Consider the map
  \[
    \xymatrix{
      \mu^*\Omega_X^d \ar[rr] \ar@{=}[d] & & \Omega_{X_\Sigma}^d \ar@{=}[d] \\
      \mu^*(\logjac_X\otimes_\ZZ\wedge^d M)  \ar[rr] & & \logjac_{X_\Sigma}\otimes_\ZZ\wedge^d M
    }
  \]
  with image $\jac_\mu\cdot\Omega_{X_\Sigma}^d=\jac_\mu\cdot\logjac_{X_\Sigma}\otimes_\ZZ\wedge^d M$. Since $\mu^*\logjac_X$, $\logjac_{X_\Sigma}$ and $\jac_\mu$ are line bundles on $X_\Sigma$, the result follows.
\end{proof}
Recall the notation of \cref{s:logjac}, let $\mathfrak{a}\subseteq\Bbbk[\monoid]$ be a nonzero monomial ideal, and consider the following subsets of $M_\RR$.
\begin{align*}
  \sigma\spcheck &= \convhull(\monoid) \\
  \mathsf{J}^{log} &= \{\phi^{gp}(\vec{e}_{i_1}+\vec{e}_{i_2}+\cdots+\vec{e}_{i_d})\mid
  \phi^{gp}(\vec{e}_{i_1})\wedge\phi^{gp}(\vec{e}_{i_2})\wedge\cdots\wedge\phi^{gp}(\vec{e}_{i_d})\neq\vec{0}\} \\
  \newt\left(\logjac_X\right) &= \convhull(\mathsf{J}^{log}+\monoid) \\
  \mathsf{J}' &= \{\phi^{gp}(\vec{u}_{j_1}^++\vec{u}_{j_2}^++\cdots+\vec{u}_{j_c}^+-\vec{1})\mid\vec{u}_{j_1}\wedge\vec{u}_{j_2}\wedge\cdots\wedge\vec{u}_{j_c}\neq\vec{0}\} \\
  \mathsf{J} &= \mathsf{J}^{log}+\mathsf{J}' \\
  \newt\left(\jac_X\right) &= \convhull(\mathsf{J}+\monoid) \\
  P &= \newt(\mathfrak{a}) = \convhull(\{\vec{m}\mid\chi^\vec{m}\in\mathfrak{a}\}) \\
  Q &= \convhull(\mathsf{J}'+\mathsf{S})=\left\{\vec{m}\mid\vec{m}+\newt\left(\logjac_X\right)\subseteq\newt\left(\jac_X\right)\right\}
\end{align*}

Let $\monoid_\sigma=\sigma\spcheck\cap M$. Thus, $X_\sigma=\spec\Bbbk[\monoid_\sigma]$ is a normal toric variety and the canonical map $X_\sigma\rightarrow X$ induced by the inclusion $\monoid\subseteq\monoid_\sigma$ is the normalization map. We say a map $\mu:X_\Sigma\rightarrow X$ from a normal toric variety is \emph{toric} if if factors as a toric map $X_\Sigma\rightarrow X_\sigma$ followed by the canonical map $X_\sigma\rightarrow X$.

\begin{thm}
  If $\mathfrak{a}\subseteq\sheaf{O}_X$ is a nonzero invariant ideal sheaf on a toric variety $X$ and $\mu:X_\Sigma\rightarrow X$ is a toric log resolution of $\mathfrak{a}\cdot\jac_X$, then
  \[
    \chi^\vec{m}\in\widehat{\multideal}(X,\mathfrak{a}^\lambda)\Leftrightarrow\vec{m}\in\interior\left(Q+\lambda P\right),
  \]
  where $\interior\left(Q+\lambda P\right)$ is the interior of the polyhedron $Q+\lambda P$.
\end{thm}

\begin{proof}
  We will mimic the proof of Howald's Theorem~\cite{MR1828466} as it appears in Cox, Little \& Schenck~\cite[Theorem~11.3.12]{MR2810322}. Let $\Sigma(1)$ be the set of rays of the fan $\Sigma$,let $\vec{n}_\rho$ be the primitive vector on the ray $\rho$ for each $\rho\in\Sigma(1)$, let $D_\rho$ be the prime divisor on $X_\Sigma$ for each $\rho\in\Sigma(1)$, and let $Z_{Y/X}=\sum_{\rho\in\Sigma(1)}a_\rho D_\rho$. Exploiting the facts that $\logjac_{X_\Sigma}=\sum_{\rho\in\Sigma(1)}D_\rho$, $\mu^*\logjac_X=\sum_{\rho\in\Sigma(1)}\min_{\vec{m}\in\mathsf{J}^{log}}\langle\vec{m},\vec{n}_\rho\rangle D_\rho$, and $J_{Y/X}=\sum_{\rho\in\Sigma(1)}\min_{\vec{m}\in\mathsf{J}}\langle\vec{m},\vec{n}_\rho\rangle D_\rho$ along with \vref{lem:k-hat}, we see
  \[
    \widehat{K}_{Y/X}-J_{Y/X}-\lfloor\lambda Z_{Y/X}\rfloor=\sum_{\rho\in\Sigma(1)}\left(-1+\min_{\vec{m}\in\mathsf{J}^{log}}\langle\vec{m},\vec{n}_\rho\rangle-\min_{\vec{m}\in\mathsf{J}}\langle\vec{m},\vec{n}_\rho\rangle-\lfloor\lambda a_\rho\rfloor\right)D_\rho.
  \]
  Thus,
  \begin{align*}
    \chi^\vec{m}\in\widehat{\multideal}(X,\mathfrak{a}^\lambda) &\Leftrightarrow \langle\vec{m},\vec{n}_\rho\rangle \geq 1-\min_{\vec{m}\in\mathsf{J}^{log}}\langle\vec{m},\vec{n}_\rho\rangle+\min_{\vec{m}\in\mathsf{J}}\langle\vec{m},\vec{n}_\rho\rangle+\lfloor\lambda a_\rho\rfloor & \text{ for all }\rho\in\Sigma(1) \\
    &\Leftrightarrow \langle\vec{m},\vec{n}_\rho\rangle > -\min_{\vec{m}\in\mathsf{J}^{log}}\langle\vec{m},\vec{n}_\rho\rangle+\min_{\vec{m}\in\mathsf{J}}\langle\vec{m},\vec{n}_\rho\rangle+\lambda a_\rho & \text{ for all }\rho\in\Sigma(1) \\
    &\Leftrightarrow \langle\vec{m},\vec{n}_\rho\rangle > \min_{\vec{m}\in\mathsf{J}'}\langle\vec{m},\vec{n}_\rho\rangle+\lambda a_\rho & \text{ for all }\rho\in\Sigma(1)
  \end{align*}
  Since the resolution $\mu:X_\Sigma\rightarrow X$ principalizes both $\logjac_X$ and $\jac_X$, it also principalizes the fractional ideal $\left(\chi^\vec{m}\mid\vec{m\in\mathsf{J}'}\right)$. Therefore, every facet of $Q+\lambda P$ is supported on a hyperplane of the form $\{\vec{m}\mid\langle\vec{m},\vec{n}_\rho\rangle=b_\rho\}$ for some $\rho\in\Sigma(1)$ and $b_\rho\in\RR$. We conclude
  \[
    \chi^\vec{m}\in\widehat{\multideal}(X,\mathfrak{a}^\lambda)\Leftrightarrow\vec{m}\in\interior\left(Q+\lambda P\right).
  \]
\end{proof}

\end{document}